\documentclass[11pt,twoside]{article}
\usepackage{amsmath, amssymb, amsfonts, amstext, amsthm, textcomp, enumerate}
\usepackage[mathscr]{euscript}
\usepackage{float}
\usepackage{booktabs}
\usepackage{mathtools}
\usepackage[left=20mm,top=0.5in,bottom=15mm]{geometry}
\usepackage{graphicx}
\usepackage{caption}
\usepackage{epstopdf}
\usepackage{longtable}
\usepackage[utf8]{inputenc}
\usepackage{color}
\usepackage{graphicx}
\usepackage{dcolumn}
\usepackage{bm}
\usepackage{epstopdf}
\usepackage[english]{babel}
\usepackage{subfigure}
\usepackage{color}

\usepackage{ulem}
\newtheorem{thm}{Theorem}[section]
\newtheorem{lem}[thm]{Lemma}
\newtheorem{defn}[thm]{Definition}

\newtheorem{rem}[thm]{Remark}
\newtheorem{ex}[thm]{Example}

\newcommand{\mnr}{\mathbf{M}_n\,(\mathbb{R})}
\newcommand{\mmn}{\mathbf{M}_{m,n}\,(\mathbb{R})}
\newcommand{\mr}{\mathbf{M}_{m,r}\,(\mathbb{R})}
\newcommand{\mnm}{\mathbf{M}_{n,m}\,(\mathbb{R})}

\newcommand{\mnnr}{\mathbf{M}_{n+1}\,(\mathbb{R})}

\newcommand{\mnrn}{\mathbf{M}_{r,n}\,(\mathbb{R})}

\renewcommand{\bar}{\overline}

\newfont{\bb}{msbm10}

\date{}

\begin{document}
\title{A short survey of $Z$-matrices and new results on the subclass of $F_0$-matrices}

\author{Samir Mondal\thanks{Department of Mathematics, Indian Institute of 
		Technology Madras,
		Chennai 600036, India (ma19d750@smail.iitm.ac.in, kcskumar@iitm.ac.in).} \ and K.C. Sivakumar$^*$}
  
\maketitle

------------------------------------------------------------------------------------------------------------------
This paper is dedicated, to Professor M.T. Nair on the occasion of his 65th
birthday. The authors wish him a long, healthy, and productive life.

------------------------------------------------------------------------------------------------------------------

	\begin{abstract}
		A real square matrix $A$ of order $n \times n~ (n \geq 3)$ is called an $F_0$-matrix, if it is a $Z$-matrix (off-diagonal entries nonpositive), all of whose principal submatrices of orders at most $n-2$ are $M$-matrices while there is at least one principal submatrix of order $n-1$, which is an $N_0$-matrix. An $M$-matrix is a $Z$-matrix  with the property that the real parts of all its eigenvalues are nonnegative. An $N_0$-matrix, in turn, is characterized by the fact that it is an invertible $Z$-matrix whose inverse is (entrywise) nonpositive. The first aim of this article is to present a short survey of some subclasses of $Z$-matrices, pertinent to the second objective, where new results concerning $F_0$-matrices are presented. 
	\end{abstract}
	


\newpage
\section{Introduction}
Let us recall the notion of an $F_0$-matrix. Let $\mnr$ denote the space of all real square matrices of order $n \times n.$ A matrix $A \in \mnr$ is called a {\it $Z$-matrix}, if all its off-diagonal entries are nonpositive.

\begin{defn}\label{f_0defn}
Let $A\in \mnr$ be a $Z$-matrix. Then $A$ is called an $F_0$-matrix if it satisfies the following conditions:\\
(1) all principal submatrices of orders at most $n-2$ are $M$-matrices.\\
(2) at least one principal submatrix of order $n-1$ is an $N_0$-matrix.
\end{defn}

This concept was introduced by \cite{gjohn1}, and the nomenclature is due to \cite{smith2}, in honour of Fan. In this article, we present some new results for this matrix class. These will be presented after providing a short literature survey of relevant subclasses of $Z$-matrices.

\section{A survey of some subclasses of $Z$-matrices}
In what follows, we undertake a short survey of some subclasses of $Z$-matrices. 

For a matrix $A \in \mnr$, we use $R(A)$ to denote its range space, $N(A)$ for the null space and $A^T,$ for its transpose. The value $\rho(.)$ denotes the spectral radius, whereas $\rho_r(B)$ stands for the maximum of the spectral radii of all principal submatrices of $B$ of order $r$. We refer to a matrix $A \in \mnr$ as a nonnegative matrix if all its entries are nonnegative. This is denoted by $A \geq 0$. $A>0$ means that all the entries of $A$ are positive. This, in particular, applies to vectors. 
In what follows, we collect a set of matrix classes, pertinent to the discussion. These classes are denoted by a number of letters of the alphabet, sometimes also going with subscripts.

A matrix $A \in \mnr$ is called a {\it $Z$-matrix}, if all the off-diagonal entries of $A$ are nonpositive. Any $Z$-matrix $A$ may be represented by $A=sI-B$, where $s$ is a real number and $B \geq 0.$ In any such representation of $A$ as above, if $s\geq \rho(B),$ then $A$ is called an {\it $M$-matrix.} A $Z$-matrix written as $A=sI-B$ is called an {\it invertible $M$-matrix}, if $z > \rho(B).$ In such a case, it is well known that $A$ is invertible and $A^{-1} \geq 0$. There is a variety of results that characterize when a $Z$-matrix is an invertible $M$-matrix. The book \cite{berpl} has fifty such necessary and sufficient conditions and serves as an excellent resource for this matrix class. Let us recall one such result. A $Z$-matrix $A$ is an invertible $M$-matrix iff there exists $x >0$ such that $Ax >0.$ It follows from this, for instance that, if $B$ is an irreducible stochastic matrix (either all the column (row) sums equal one), and if $A$ is the $Z$-matrix given by $A=sI-B$ with $s>1$, then $A$ is an invertible $M$-matrix. On the other hand, if $A=I-B,$ then $A$ is a singular $M$-matrix. This is due to the fact that $1$ is a simple eigenvalue of $B$ (by the Perron-Frobenius theorem) and that $e,$ the "all entries one" vector (whose dimension will be clear from the context), is an eigenvector of $B$, so that $e$ is a nonzero vector in $N(A)$. 

For $A \in \mathbb{C}^{n \times n}$, let ${\cal M}_A$ denote the {\it comparison matrix} defined by 
\begin{equation}
 ({\cal M}_A)_{ij} = \begin{cases}
  ~~|a_{ii}|, & ~{\textit if} ~i = j,\\
  -|a_{ij}|, & ~{\textit if}~i\neq j.
  \end{cases} \\
\end{equation}
Note that ${\cal M}_A$ is a $Z$-matrix. $A$ is called an invertible $H$-matrix, if ${\cal M}_A$ is an invertible $M$-matrix. Invertible $H$-matrices have been extensively studied in the literature. A celebrated result for invertible $H$-matrices is the following, proved in \cite{ost}. 

\begin{thm}\label{ostr}
Let $A$ be an invertible $H$-matrix. Then the following coordinate wise inequality holds:
$$|A^{-1}| \leq {{\cal M}_A}^{-1}.$$
\end{thm}

The book \cite{neu} is a good source for results on invertible $H$-matrices. We refer the reader to \cite{samirtsatkcs} for a number of recent results on invertible $H$-matrices, which in turn, are motivated by those of $M$-matrices.  

A nonnegative and nonsingular matrix is called an {\it inverse $M$-matrix}, if its inverse is an $M$-matrix. Note that, since an invertible $M$-matrix has the property that its inverse is nonnegative, we need this requirement of nonnegativity, for a matrix to qualify to be an inverse $M$-matrix. We refer the reader to \cite{john}, \cite{Johnsmith} and the monograph \cite{jst} for more details on inverse $M$-matrices. We refer the reader also to the recent work \cite{manamiinterval}.

If a matrix $A$ is such that $A-I$ is an invertible $M$-matrix, then $A$ is invertible, and both $A$ and $I-A^{-1}$ are invertible $M$-matrices. This result is due to Fan. This, as well as its converse were proved in \cite[Theorem 5.1]{manamiinequalities} and \cite[Proposition 5.1]{kalkcs1}. Verbatim analogues of this result for singular $M$-matrices were proved in \cite{kalkcs1}, while the case of inverse $M$-matrices was studied in \cite[Theorem 2.1]{kalkcs2}.
Results of this type for $H$-matrices are proved in \cite{manamiinequalities}. 

Next, let $A \in \mnr$ be a $Z$-matrix given by $A=sI-B$, with $B\geq 0$. Let $\rho_{n-1}(B)$ denote the maximum of the spectral radii of the principal submatrices of $B$ of order $(n-1) \times (n-1)$. If $s$ satisfies the inequalities $\rho_{n-1}(B) < s < \rho(B),$ then $A$ is referred to as an {\it $N$-matrix}. This matrix class was introduced and investigated in \cite{fan}. For instance, it is shown that an $N$-matrix is precisely a $Z$-matrix with the property that its determinant is negative, with each proper principal submatrix being an  invertible $M$-matrix. Let us recall that a proper principal submatrix of a matrix is any principal submatrix other than the matrix itself. In \cite{john}, a procedure of constructing an $N$-matrix, whose leading maximal principal submatrix being an invertible $M$-matrix, was presented. Let us include it here, as part of this exposition. Let $A$ be an invertible $M$-matrix so that $A^{-1} \geq 0$. Set $q:=e^TA^{-1}e$, the sum of all the entries of $A^{-1}$, so that $q>0.$ Set $t$ to be the minimum of all the column sums of $A^{-1}$ so that $q>t$. Further, set $\gamma=-\frac{1}{q-t}$ and $\delta=-\frac{1}{q},$ so that $\gamma < \delta <0.$ Finally, choose $\alpha \in (\gamma, \delta).$ If 
$$M=\begin{pmatrix}
A & -e\\
\alpha e^t & 1
\end{pmatrix},$$
then $M$ is an invertible $Z$-matrix, ($M^{-1} <0$) and $M$ is an $N$-matrix. For more details, we refer the reader to the discussion after Theorem 2.10, \cite{john}.

Next, let us consider an extension of  $N$-matrices. Let $A \in \mnr$ be a $Z$-matrix given by $A=sI-B$, with $B\geq 0$. If $s$ satisfies the inequalities $\rho_{n-1}(B) \leq  s < \rho(B),$ then $A$ is referred to as an {\it $N_0$-matrix}.  This matrix class was proposed, and several interesting properties were proved, in \cite{gjohn1}. Note that every $N$-matrix is an $N_0$-matrix and that the latter matrix class is the topological closure of the former. Let us include an interesting extension of the property of $N$-matrices, mentioned earlier, to the class of $N_0$-matrices. A matrix $A$ is an $N_0$-matrix precisely if $A$ is a $Z$ matrix with the property that all proper principal submatrices are (not necessarily invertible) $M$-matrices, with the determinant of $A$ being negative \cite[Lemma 2.1]{gjohn1}. Here is another result: If $A$ is a $Z$-matrix, then $A$ is an $N_0$-matrix iff $A^{-1} \leq 0$ and is irreducible \cite[Theorem 2.7]{gjohn1}. Note that a nonsingular matrix is irreducible iff its inverse is irreducible. The following result concerning the spectrum of an $N_0$-matrix is also interesting. Let $A$ be an $N_0$-matrix. Then $A$ has exactly one negative eigenvalue \cite[Theorem 2.1]{john}. For further results, we refer the reader to \cite{smith1}.

Let us turn our attention to the inverse class of $N_0$-matrices. A nonpositive nonsingular matrix $A$ is called an {\it inverse $N_0$-matrix}, if $A^{-1}$ is an $N_0$-matrix. This matrix class was introduced in \cite{gjohn2} and some characterizations were proved. We shall be interested in a sufficient condition. In order to state this result, let us recall the following: Let real numbers $a_1,a_2, \ldots, a_n$ be such that $a_n > a_{n-1} > \ldots > a_1.$ A matrix $A:=(a_{ij})$ is called a {\it matrix of type $D$}, if 
\begin{equation}
 a_{ij} = \begin{cases}
  a_i, & ~{\textit if} ~i \leq j,\\
  a_j, & {\textit{otherwise}}.
  \end{cases} \\
\end{equation}

A matrix of type $D$ of order $n \times n$ will be denoted by $D_n.$ For instance, 
  $$D_4=
  \begin{pmatrix}
   a_1 & a_1 & a_1 & a_1\\
   a_1 & a_2 & a_2 & a_2\\
   a_1 & a_2 & a_3 & a_3\\
   a_1 & a_2 & a_3 & a_4
  \end{pmatrix}.$$
We also have the following recurrence relation:
$$D_{n+1}=
  \begin{pmatrix}
      D_n & b \\
      b^T & a_{n+1}
  \end{pmatrix},$$ where $b=(a_1,a_2, \ldots, a_n)^T.$ Type $D$ matrices were introduced in \cite{markham}, where the following result was shown. 
  
\begin{thm}\cite[Theorem 2.3]{markham}
Let $A$ be a type $D$-matrix with $a_1 >0.$ Then $A$ is an inverse $M$-matrix with $A^{-1}$ being tridiagonal.  
\end{thm}

An analogue for inverse $N_0$-matrices was also proved.

\begin{thm} \cite[Theorem 2.4]{gjohn2}
Let $A$ be a type $D$-matrix with $a_n <0.$ Then $A$ is an  inverse $N_0$-matrix such that $A^{-1}$ is tridiagonal.
\end{thm}

Let us turn our attention to the main object of investigation, viz., $F_0$-matrices (see Definition \ref{f_0defn}). Let's recall a characterization of $F_0$-matrices.

\begin{thm}\cite[Theorem 2.3]{chen}
 The matrix $A$ is an $F_0$-matrix if and only if $A=tI-B,$ with $B\geq 0$ and $\rho_{n-2}(B)\leq t < \rho_{n-1}(B)$ for $n\geq 3.$ 
\end{thm}

Let $M\in\mnnr$ be an $F_0$-matrix. Then, one may assume, without loss of generality (by permuting the rows and columns of $M$, if necessary) that $M$ may be written as $$M=\begin{pmatrix}
A & b\\
c^T & d
\end{pmatrix},$$ where $A\in \mnr$ is an $N_0$-matrix, $b, c \in \mathbb{R}^n$ are nonpositive and $d \in \mathbb{R}$ is nonnegative. We will be making frequent use of this representation.

Here is an interesting result for this matrix class.

\begin{thm} \cite[Theorem 3.1]{john} \label{f0mot}
Let $A$ be a nonsingular $F_0$-matrix. Then $\rm det(A) <0$ and $A^{-1}$ is a $Z$-matrix with at least one positive diagonal entry. 
\end{thm}

The converse question was considered in \cite{chen}.

\begin{thm}\cite[Theorem 2.4]{chen}\label{charf0}
Suppose that $M$ is a nonsingular $Z$-matrix. Then $M\in F_0$ if and only if \\
$(1)$ ${\rm det}(M) <0,$\\
$(2)$ all principal minors of $M^{-1}$ of order at least two are nonpositive,\\
$(3)$ $M^{-1}$ has at least one positive diagonal entry.
\end{thm}

There is much more information on nonsingular $F_0$-matrices. We only state one relevant result. For more details, we refer to \cite{smith2}.

\begin{thm} \cite[Lemma 4.1]{smith2}
Let $A$ be a nonsingular $F_0$-matrix. Then all principal minors of $A^{-1}$ of order at least two, are nonpositive.
\end{thm}

Next, let us recall some results that are true for circulant matrices of order $3 \times 3$. Let $$P=\begin{pmatrix}
 0 & 1 & 0\\
 0 & 0 & 1 \\
 1 & 0 & 0
\end{pmatrix}$$ and let $\cal C$ denote the class of all ($3\times 3$) circulant matrices so that each $A \in \cal C$ is of the form 
\begin{equation}\label{circulantmatrix}
\alpha_0I +\alpha_1 P+\alpha_2 P^2,
\end{equation}
where $\alpha_0, \alpha_1$ and $\alpha_2$ are real numbers. 

The class $\cal C$ has the property that if $A \in \cal C$ and $\alpha \in \mathbb{R},$ then $\alpha I+A, \alpha A \in \cal C.$ If $A \in \cal C$ is nonsingular, then $A^{-1} \in \cal C.$ Let $A \in \cal C$ have the form as above. Then $\alpha_0+\alpha_1+\alpha_2$ is an eigenvalue of $A$ with $(1,1,1)^T$ as a corresponding eigenvector. Next, let $\omega =-\frac{1}{2}+i \frac{\sqrt{3}}{2},$ with $i^2=-1.$ Then $z_A:=\alpha_0 +\alpha_1 \omega +\alpha_2 \omega^2$ is an eigenvalue of $A,$ with $(1, \omega, \omega^2)^T,$ as an associated eigenvector. Of course, the other eigenvalue of $A$ is $\bar{z_A}.$

Let ${\cal C}^{(1)}$ and ${\cal C}^{(-1)}$ denote two subclasses of $\cal C$ defined by:
\begin{equation}
   {\cal C}^{(1)}:=\{A \in {\cal C}:\alpha_0+\alpha_1+\alpha_2=1 \} 
\end{equation}
and 
\begin{equation}
  {\cal C}^{(-1)}:=\{A \in {\cal C}:\alpha_0+\alpha_1+\alpha_2=-1 \}  
\end{equation}

It is shown in \cite[Theorem 4.2]{fiedmarneu} that if $A \in {\cal C}^{(1)}$, then $A$ is uniquely determined by $z_A.$ More precisely, any $A$ of the form \ref{circulantmatrix}, satisfies $A \in {\cal C}^{(1)}$ iff $z_A=\alpha_0 +\alpha_1 \omega +\alpha_2 \omega^2,$ whenever  $\alpha_0 +\alpha_1+\alpha_2=1.$ 
In particular, in the same result, it is shown that $A\in {\cal C}^{(1)}$ is an $M$-matrix iff $z_A=a+ib$ satisfies the inequality $$a\geq 1+\sqrt{3}|b|.$$

Now, let us turn our attention to matrices in $\cal C$ which have a negative eigenvalue. For an $N_0$-matrix, here is a result that was proved in \cite[Theorem 3.1]{smith2}. $A\in{\cal  C}^{(-1)}$ is uniquely determined by $z_A.$  In particular, $A\in {\cal C}^{(-1)}$ is an $N_0$-matrix iff $z_A=a+ib$ satisfies the inequalities 
\begin{center}
$a\geq \rm max \{\frac{1}{2}, -1 +\sqrt{3}|b|\}$ and  $(a-1)^2+b^2\geq 1.$
\end{center}

We obtain a version of these results for the classes  $F_0$-matrices, inverse $F_0$-matrices, inverse $N_0$-matrices and inverse $M$-matrices, in Theorem \ref{circ}.

A number of results stated here for the specific subclasses, have generalizations to the larger class of nonsingular $Z$-matrices. We refer the reader to the works \cite{fiedlermark}, \cite{nabben1}, \cite{nabben2}. For some recent results see \cite{nabben3} and for a short survey, refer to \cite{stuart}.

We close this section by giving an overview of the new results of this article. In Theorem \ref{singularf0}, we show that the group inverse of a reducible $F_0$-matrix, is nonpositive. For irreducible $F_0$-matrices of order $3 \times 3$, we obtain the same conclusion in Theorem \ref{3x3groupinv}. While an analogoue of Theorem \ref{singularf0} is shown to be false for the Moore-Penrose inverse, once again, affirmative answers are obtained for matrices of order $3 \times 3$. These are proved in Theorem \ref{3x3mpinv_red} and Theorem \ref{3x3mpinv_irred}. Some specific results relevant to the linear complementarity problem are collected in Theorem \ref{compl}. Results on $3 \times 3$ circulant matrices pertinent to the matrix classes considered here, are presented in Theorem \ref{circ}. Finally, certain well known invertible matrices are shown to be inverse $F_0$-matrices in Theorem \ref{d_n+1f0}.

\section{Preliminaries}
Let $A \in \mmn$. Recall that the {\it Moore-Penrose inverse} of $A$ is the unique matrix $X\in \mnm$ that satisfies the equations $AXA=A, XAX=X, (AX)^*=AX$ and $(XA)^*=XA.$ We denote the Moore-Penrose inverse of $A$ by $A^{\dag}$. The {\it group inverse} of a matrix $A\in \mnr$ is the unique matrix $X \in \mnr$, if it exists, that satisfies the equations $AXA=A, XAX=X$ and $AX=XA$. If it exists, then the group inverse is denoted by $A^{\#}$. A necessary and sufficient condition for the group inverse of a matrix $A$ to exist is that $R(A^2)=R(A)$, which is equivalent to the condition $N(A^2)=N(A)$. Another characterization is that $A^{\#}$ exists iff $R(A)$ and $N(A)$ are complementary subspaces of $\mathbb{R}^n$. It is easy to show that a nilpotent matrix does not have group inverse, whereas the group inverse of an idempotent matrix (exists and) is itself. 

Next, we recall the full rank factorization. Let $A \in \mmn$ with $\rm rk (A)=r>0.$ We say that $A$ has a full rank factorization, if there exist $F\in \mr, G\in \mnrn$ such that $\rm rk(F)=\rm rk(G)=r$ and $A=FG$. It is well known that any nonzero matrix has a full rank factorization. Let $A=FG$ be a full rank factorization of $A$. Then $A^{\dag}=G^{\dag}F^{\dag}=G^*(GG^*)^{-1}(F^*F)^{-1}F^*.$ Again, let $A=FG$ be a full rank factorization of $A\in \mnr.$ Then, a necessary and sufficient condition for $A^{\#}$ to exist is that $GF$ is invertible. In such a case, one also has the formula $A^{\#}=F(GF)^{-2}G.$ For further details on these generalized inverses, we refer the reader to \cite{bg}.

We will have occasion to discuss relationships between a matrix class and the linear complementarity problem. Let us recall this notion. Let $A \in \mnr$ and $q \in \mathbb{R}^n.$ {\it The Linear Complementarity Problem} LCP$(A,q)$ is to determine if there exists $x\in \mathbb{R}^n$ such that 
$$x \geq 0, ~y:=Ax+q \geq 0~ \textit{and} ~x^Ty=0.$$ 

A vector $x$ that satisfies the first two conditions is called a {\it feasible} vector, and the set of all feasible vectors for LCP$(A, q)$ is denoted by FEA$(A, q)$. A vector $x$ that satisfies all the three conditions is called a (complementary) {\it solution}, and the set of all solutions for LCP$(A, q)$ is denoted by SOL$(A, q)$.

A matrix $A$ is called a $Q$-matrix if for any $q \in \mathbb{R}^n$, the linear complementarity problem LCP$(A, q)$ has at least one solution. A matrix $A$ is called a $Q_0$-matrix if whenever a feasible vector exists for LCP$(A, q)$, there exists at least one complementary solution. That is, for any $q \in \mathbb{R}^n$ such that FEA$(A, q) \neq \emptyset$, we must have SOL$(A, q) \neq \emptyset$. 

A matrix $A \in \mnr$ is said to be an {\it $R_0$-matrix}, if the LCP$(A,0)$ has zero as the only solution. This is equivalent to the implication:
$$x \geq 0, Ax \geq 0 ~{\textit and} ~x^TAx=0 \Longrightarrow x=0.$$

Matrix $A$ is called {\it semimontone} if 
$$0\neq x \geq 0 \Longrightarrow ~{\textit there~ exists}~ k~{\textit such ~that}~ x_k >0 ~{\textit and}~ (Ax)_k \geq 0.$$

\section{Main Results}
\subsection{Nonpositivity/$Z$-property of generalized inverses of $F_0$-matrices}

As mentioned earlier, if $M\in\mnnr$ is an $F_0$-matrix,  then, $M$ may be written as $$M=\begin{pmatrix}
A & b\\
c^T & d
\end{pmatrix},$$ where $A\in \mnr$ is an $N_0$-matrix, $b, c \in \mathbb{R}^n$ are nonpositive and $d \in \mathbb{R}$ is nonnegative. In general, $F_0$-matrices may be singular and possess some interesting nonnegativity/nonpositivity properties. One of the objectives of this survey is to prove some of those. An illustrative example gives a hint at what we are interested in. Let $A=-\begin{pmatrix}
    0 & 1\\
    1 & 0
\end{pmatrix},~ b=c^T=(-1,0)^T$ and $d=0.$ Then $A$ is an $N_0$-matrix, $b,c$ are nonpositive vectors and 
$$M=\begin{pmatrix}
~~0 & -1 & -1\\
-1 & ~~0 & ~~0\\
-1 & ~~0 & ~~0
\end{pmatrix}$$ is a singular $F_0$-matrix. It is easy to verify that $M^{\#}$ exists and that $M^{\#}=\frac{1}{2}M.$ Thus, $M^{\#}$ is a $Z$-matrix. Further, it is important to observe in this simple example the fact that one may also have all the diagonal entries, of the group inverse, to be zero. 

Motivated by this example, we may ask the question: Does Theorem \ref{f0mot} have a generalization for singular matrices, which however, are group invertible? We show that the answer is in the affirmative. This is presented in Theorem \ref{singularf0}.

Let us make some preliminary observations. Let $M$ be as above. It is easy to observe that, for $M$ to be singular, it is necessary and sufficient that the number $d$ satisfies $d=c^TA^{-1}b.$ Due to the nonpositivity of the off-diagonal entries and the matrix $A^{-1},$ it follows that $d \leq 0.$ This, in conjunction with the fact that all diagonal entries of an $F_0$-matrix are nonnegative, implies that $d=0.$ Thus, a singular $F_0$-matrix is of the form $$M=\begin{pmatrix}
A & b\\
c^T & 0
\end{pmatrix},$$
with the additional assumptions on $b,c$ and $A$ as above. In our first result, we show that the group inverse of $M$ exists. This is rather interesting and is a consequence of the nonnegativity/nonpositivity assumptions on the relevant  matrix/vectors. We also show that the group inverse of $M$ is a $Z$-matrix, if it is reducible. Thus, we have an extension of Theorem \ref{f0mot} for the case of singular matrices. We also prove some results relating to the LCP. 

We will make use of the fact that if $A$ is an invertible matrix, then a necessary and sufficient condition for the (rank one perturbed) matrix $A+uv^T$ to be invertible (given by the Sherman-Morrison-Woodbury formula), is the statement that $1+v^TA^{-1}u \neq 0.$ In such a case, we also have the explicit expression for the inverse given by: $$(A+uv^T)^{-1}=A^{-1}-\frac{1}{1+v^TA^{-1}u}A^{-1}uv^TA^{-1}.$$

\begin{thm}\label{singularf0}
Let $M$ be an $F_0$ matrix, with the representation given as above. Let $M$ be singular. We then have the following:\\
(1) $\rm rk(M)=n.$\\
(2) $M^{\#}$ exists.\\
(3) Let $M$ be reducible. Then $M^{\#} \leq 0$ (so that $M^{\#}$ is a $Z$-matrix). 
\end{thm}
\begin{proof}
(1) This is easy to see, as $M$ is singular, while the leading maximal principal submatrix $A$, is nonsingular.\\
(2) We have 
$$M=\begin{pmatrix}
A & b\\
c^T & 0
\end{pmatrix}=FG=:\begin{pmatrix}
A\\
c^T \end{pmatrix}\begin{pmatrix}
    I & A^{-1}b.
\end{pmatrix}$$
Then $$GF=\begin{pmatrix}
    I & A^{-1}b \end{pmatrix}\begin{pmatrix}
    A\\
c^T 
\end{pmatrix}=A+gc^T,$$
where $g:=A^{-1}b.$ Note that 
$$1+c^TA^{-1}g=1+c^TA^{-2}b > 0,$$
since $A^{-2}=(A^{-1})^2 \geq 0$ and $b,c$ are nonpositive vectors. Thus, by the Sherman-Morrison-Woodbury formula, $GF$ is invertible and so $M^{\#}$ exists.\\
(3) Let $M$ be reducible. Since $A$ is an $N_0$-matrix, it is irreducible (as mentioned earlier) and so it follows that $c=0$. From the full rank factorization given earlier, we now have $GF=A$ (which is invertible) so that $$M^{\#}=F(GF)^{-2}G=FA^{-2}G=\begin{pmatrix}
A\\
0 \end{pmatrix}A^{-2}\begin{pmatrix}
    I & A^{-1}b
\end{pmatrix}=\begin{pmatrix}
    A^{-1} & A^{-2}b\\
    0 & 0
\end{pmatrix}.$$ 
Now, $A^{-1}$ is a nonpositive matrix, $A^{-2}$ is a nonnegative matrix and $b$ is a nonpositive vector and so $M^{\#} \leq 0$. In particular, $M^{\#}$ is a $Z$-matrix.
\end{proof}

In relation to item (3), one may ask as to what conclusion one could draw, when the matrix is irreducible. A number of randomly generated numerical examples have all turned out to show that the group inverse of a singular irreducible matrix is also nonpositive. However, a general proof seems elusive. Nevertheless, we are able to prove this result for irreducible matrices of order $3 \times 3$. 

First, in the next auxiliary result, we identify $N_0$-matrices of order $2 \times 2.$ The proof is trivial. We use this to show that all $F_0$-matrices of order $3 \times 3$ (reducible or irreducible) have the property that their group inverses are nonpositive.

\begin{lem}\label{aux}
Let $A \in \mathbb{R}^{2 \times 2},$ be an $N_0$-matrix. Then $A$ is of one of the four forms given by: 
\begin{equation}\label{form1}
A=\begin{pmatrix}
  0 & \beta\\
  \gamma & \delta
\end{pmatrix},~{\textit with}~\delta >0, ~{\textit and}~\beta, \gamma <0. 
\end{equation}
\begin{equation}
A=\begin{pmatrix}\label{form2}
  \alpha & \beta\\
  \gamma & 0
\end{pmatrix}, ~{\textit with}~ \alpha >0~{\textit{and}~} \beta, \gamma <0. 
\end{equation}
\begin{equation}\label{form3}
A=\begin{pmatrix}
  0 & \beta\\
  \gamma & 0
\end{pmatrix}, ~{\textit{with}}~ \beta, \gamma <0.  
\end{equation}
\begin{equation}\label{form4}
A=\begin{pmatrix}
  \alpha & \beta\\
  \gamma & \delta
\end{pmatrix}, ~{\textit{with}}~ \alpha, \delta >0, ~\beta, \gamma <0~ {\textit{and}}~ \alpha\delta< \beta \gamma.  
\end{equation}
\end{lem}

\begin{thm}\label{3x3groupinv}
Let $M\in \mathbb{R}^{3 \times 3}$ be a singular irreducible $F_0$-matrix. Then $M^\#\leq 0.$
\end{thm}
\begin{proof}
Let $M=\begin{pmatrix}
 A & b \\
 c^T & 0
\end{pmatrix},$ where $A$ is an $N_0$-matrix, and $b, c$ are nonpositive column vectors. First, we assume that $b,c$ have at least one nonzero component. Otherwise, $M$ (or $M^T$) is reducible and this has been considered in (3) of Theorem \ref{singularf0}. 

We make use of Lemma \ref{aux} for the four forms that the matrix $A$ can take. 
Let us begin with $A$ as in \ref{form4}. Then $A^{-1}
=\frac{1}{\alpha \delta - \gamma \beta}
\begin{pmatrix}
~\delta & -\beta \\
-\gamma & ~\alpha
\end{pmatrix}$ and so, the condition that $c^TA^{-1}b=0$ translates to $$c_1\delta b_1-c_1 \beta b_2-c_2 \gamma b_1+c_2 \alpha b_2=0.$$
Since each term on the left hand side is nonnegative, each must equal zero. So, either $c=0$ or $b=0.$ Thus, once again, $M$ is reducible. 

Next, let $A$ have form \ref{form3}. Then $A^{-1}=\begin{pmatrix}
0 & \frac{1}{\gamma} \\
\frac{1}{\beta} & 0
\end{pmatrix}$ and so, 
$$0=c^TA^{-1}b=\frac{1}{\gamma}c_1b_2+\frac{1}{\beta}b_1c_2.$$ The right hand side is a sum of nonnegative terms and so each equals zero. So, $c_1b_2=0=b_1c_2.$ This results in the following two possibilities, employing irreducibility.
\begin{center}
$M=\begin{pmatrix}
    0 & \beta & b_1\\
    \gamma & 0 & 0\\
    c_1 & 0 & 0
\end{pmatrix}$ and 
$M=\begin{pmatrix}
    0 & \beta & 0\\
    \gamma & 0 & b_2 \\
    0 & c_2 & 0 
\end{pmatrix}.$    
\end{center}
In the first case, one has 
$$M^{\#}=\frac{1}{\beta\gamma+b_1c_1}\begin{pmatrix}
    \gamma & \beta & b_1\\
    \gamma & 0 & 0\\
    c_1 & 0 & 0
\end{pmatrix} \leq 0$$
and in the latter, we have
$$M^{\#}=\frac{1}{\beta\gamma+b_2c_2}M \leq 0.$$

Finally, we note that a matrix $A$ given as in \ref{form2} is permutationally equivalent to \ref{form1} and so, it suffices to consider only a matrix of type \ref{form1}. Then $A^{-1}=-\frac{1}{\beta\gamma}\begin{pmatrix}
~~\delta & -\beta \\
-\gamma & ~~0
\end{pmatrix}$ and so, again using $c^TA^{-1}b=0$, we have $$\delta c_1b_1- \beta c_1b_2-\gamma b_1c_2=0.$$
As before, each term on the right hand side equals zero. Taking into account the irreducibility of the matrix, we have:
$$M=\begin{pmatrix}
    0 & \beta & 0\\
    \gamma & \delta & b_2 \\
    0 & c_2 & 0 
\end{pmatrix}.$$

It is easy to verify that for this case,
$$M^\#
=\frac{1}{(\beta \gamma +b_2 c_2)^2}
\begin{pmatrix}
  -\beta \gamma \delta & \beta (\beta \gamma +b_2 c_2) & -\beta \gamma \delta \\
  ~\gamma (\beta \gamma +b_2 c_2) & 0 &  ~ b_2 (\beta \gamma +b_2 c_2)\\
  -\gamma \delta c_2 & c_2 (\beta \gamma +b_2 c_2) & \delta b_2 c_2
\end{pmatrix}\leq 0,$$
completing the proof.
\end{proof}

Next, we turn our attention to the case of the Moore-Penrose inverse. A natural question in relation to item (3) of Theorem \ref{3x3groupinv}, is whether the Moore-Penrose inverse of $M$ is a $Z$-matrix. We show that this is not the case, even when $M$ is irreducible. 

\begin{ex}\label{mpcounter}
Let $$M=
\begin{pmatrix}
    ~~2 & -4 & -1 & -1 \\
    -2 &  ~~4 & -1 & -4 \\
    -2 & -2 & ~~1 & -6\\
    ~~0 & ~~0 & ~~0 & ~~0
\end{pmatrix}.$$ Here $A=
\begin{pmatrix}
    ~~2 & -4 & -1\\
    -2 &  ~~4 & -1\\
     -2 & -2 & ~~1
    \end{pmatrix},$ is a $Z$-matrix, $A^{-1}$ is irreducible and $A^{-1}\leq 0.$ So, $A$ is an $N_0$-matrix. All the diagonal entries and all the $2 \times 2$ principal minors of $M$ are nonnegative. Hence, $M$ is an $F_0$-matrix. Now, $$M^\dagger=\frac{1}{2609}
\begin{pmatrix}
    ~~538 & ~~245 & -327 & 0\\
    -149 & ~~257 & -220 & 0 \\
    -692 & -577 & ~~440 & 0\\
    -245 & -291 & -176 & 0
\end{pmatrix}$$ has a positive off-diagonal entry and so $M^\dagger$ is not a $Z$-matrix.
\end{ex}

Since, in general, $M^{\dag}$ is not a nonpositive matrix, one may ask whether $M^{\dag}$ is nonpositive on $R(M).$ In the next example, we show that even this not true. Let us make this precise. 

A matrix $G \in \mnr$ is said to be {\it row-monotone} if $$Gx \geq 0, ~x \in R
(G^T) \Longrightarrow x \geq 0.$$
Then, $G$ is row-monotone iff (\cite[Lemma 2.4]{debkcs}, \cite{berpl1})
$$x \in R(G) \cap \mathbb{R}^n_+ \Longrightarrow G^{\dag}x \geq 0.$$
This generalizes the requirement of the nonnegativitiy of $G^{\dag}$, to $G^{\dag}$ being nonnegative on its range space $R(G)$. Specifically, if $G^{\dag} \geq 0,$ then $G$ is row-monotone, while the converse is false. 

\begin{ex}
Let $M$ be as in Example \ref{mpcounter}. The vector $(37,14,30,-12)^T$ spans the null space $N(M).$ So, if $(x_1,x_2,x_3,x_4)^T=x \in R(M^T),$ then we have $$37x_1+14x_2+30x_3-12x_4=0.$$ Thus, the vector $u:=(6,3, -6,7)^T \in R(M^T)$, has both positive and negative coordinates, and $Mu=-(1,22,66,0)^T \leq 0$. Thus, $-M$ is not row-monotone.
\end{ex}

Let us also record the following "monotonicity type" affirmative result, implied by the inverse nonpositivity of the leading principal submatrix $A$. 

\begin{thm}
Let $M=\begin{pmatrix}
 A & b \\
 c^T & 0
\end{pmatrix} \in \mathbb{R}^{(n+1) \times (n+1)}$. Then the following implication holds:
$$Mx \leq 0, ~x_{n+1} =0 \Longrightarrow x\geq0.$$
\end{thm}
\begin{proof}
Let $x=(w,0)^T$ be such that $0 \geq Mx=\begin{pmatrix}
 A & b \\
 c^T & 0
\end{pmatrix} \begin{pmatrix}
 w \\
 0
\end{pmatrix}=\begin{pmatrix}
Aw \\
c^Tw
\end{pmatrix}.$ Then $Aw\leq 0,$ which together with $A^{-1} \leq 0$ implies that $w \geq 0.$ Thus $x\geq 0.$
\end{proof}

Let us illustrate the above result for the matrix of Example \ref{mpcounter}.

\begin{ex}
Let $M$ be as in Example \ref{mpcounter} and $x=(x_1,x_2,x_3,0)^T$. Then $Mx \leq 0$ unpacks into $2x_1-4x_2-x_3 \leq0, -2x_1+4x_2-x_3\leq 0$ and $-2x_1-2x_2+x_3 \leq 0.$ Adding the first and the second inequality results in $x_3 \geq0,$ the sum of the first and the third yield $x_2 \geq 0$ while summing the second and the third give $-4x_1+2x_2 \leq 0,$ so that $x_1 \geq 0.$ We have shown that $x \geq 0.$   
\end{ex}

Interestingly, for the case of $3\times 3$ matrices, as in the case of the group inverse (which was even nonpositive), we show that the Moore-Penrose inverse of $M$ is a $Z$-matrix. We treat the cases of reducible and irreducible matrices, separately.

\begin{thm}\label{3x3mpinv_red}
Let $M\in \mathbb{R}^{3 \times 3}$ be a singular reducible $F_0$-matrix. Then $M^{\dag}$ is a $Z$-matrix.
\end{thm}
\begin{proof}
Since the transpose of a $Z$-matrix is again a $Z$-matrix, there is no loss of generality in assuming that  $M=\begin{pmatrix}
 A & b \\
 0 & 0
\end{pmatrix},$ where $A$ is an $N_0$-matrix and $b$ is a nonpositive column vector. By using the Greville's method (\cite{bg}, \cite{kcsjota}) (which appears to be more suitable than the full rank factorization, in this case) it may be shown that 
$$M^{\dag}=\frac{1}{1+{\parallel g 
\parallel}^2} \begin{pmatrix}
 ({1+{\parallel g 
\parallel}^2})A^{-1}-gg^TA^{-1}& 0 \\
g^TA^{-1}  & 0
\end{pmatrix},$$
where $g:=A^{-1}b \geq0.$ Note that since $g^TA^{-1} \leq 0$ the bottom-left row vector is nonpositive. We show that the top
left (leading principal) submatrix viz., $$(({1+{\parallel g 
\parallel}^2})I-gg^T )A^{-1},$$ is a $Z$-matrix, for all the four possible forms that the matrix $A$ can assume. This means that we must prove that 
\begin{equation*}
  \langle (({1+{\parallel g 
\parallel}^2})I-gg^T )A^{-1}e^1, e^2 \rangle  
\end{equation*}
and 
\begin{equation*}
   \langle (({1+{\parallel g 
\parallel}^2})I-gg^T )A^{-1}e^2, e^1 \rangle 
\end{equation*}
are nonpositive. Note that for any two vectors $u,v$, we have  $$\langle (({1+{\parallel g 
\parallel}^2})I-gg^T )A^{-1}u, v \rangle=(1+{\parallel g \parallel}^2) \langle A^{-1}u,v \rangle -(g^TA^{-1}u)\langle g,v\rangle,$$
where we have deliberately used both the dot product vectors as well as the inner product notation, for convenience. Thus, we need to prove that the numbers 
\begin{equation}\label{21position}
 (1+{\parallel g \parallel}^2) \langle A^{-1}e^1,e^2 \rangle -(g^TA^{-1}e^1)\langle g,e^2\rangle    \end{equation}
 and 
 \begin{equation}\label{12position}
  (1+{\parallel g \parallel}^2) \langle A^{-1}e^2,e^1 \rangle -(g^TA^{-1}e^2)\langle g,e^1\rangle  \end{equation}
are both nonpositive. 

Let $A=\begin{pmatrix}
    \alpha & \beta \\
    \gamma & \delta
\end{pmatrix}$, where $\alpha, \delta$ are nonnegative, $\beta,\gamma$ are nonpositive, and $b^T=(b_1, b_2)$ with $b_1, b_2\leq 0.$ We have
$A^{-1}=\frac{1}{\rm det A}\begin{pmatrix}
    \delta & -\beta \\
    -\gamma & \alpha
\end{pmatrix}$ and so 
$$A^{-1}e^1=\frac{1}{\rm det A}\begin{pmatrix}
    \delta\\-\gamma
\end{pmatrix}~ \textit{and}~A^{-1}e^2=\frac{1}{\rm det A}\begin{pmatrix}
    -\beta\\
    \alpha
\end{pmatrix}.$$
Thus, $$g^T=A^{-1}b=\frac{1}{{\rm det}A}(\delta b_1-\beta b_2, \alpha b_2-\gamma b_1).$$  It may be verified that $$(1+{\parallel g \parallel}^2) \langle A^{-1}e^2,e^1 \rangle=\frac{-\beta}{({\rm det}A)^3}\{(\alpha b_2-\gamma b_1)^2+(\beta b_2-\delta b_1)^2+({\rm det}A)^2\},$$
Now,
\begin{eqnarray*}
 (1+{\parallel g \parallel}^2) \langle A^{-1}e^2,e^1 \rangle -(g^TA^{-1}e^2)\langle g,e^1\rangle &=&  \frac{(\alpha b_2-\gamma b_1)(\beta \gamma -\alpha \delta)}{({\rm det}A)^3}-\frac{\beta}{{\rm det}A} \\
 &=& -\frac{b_1(\alpha b_2-\gamma b_1)}{({\rm det}A)^2}-\frac{\beta}{{\rm det}A} 
 \leq 0.
 \end{eqnarray*}
Similarly, one may verify that 
 \begin{eqnarray*}
   (1+{\parallel g \parallel}^2) \langle A^{-1}e^1,e^2 \rangle -(g^TA^{-1}e^1)\langle g,e^2\rangle &=& \frac{b_2(\beta b_2- \delta b_1)(\alpha \delta -\beta \gamma)}{({\rm det}A)^3} - \frac{\gamma}{{\rm det}A}\\
   &=& \frac{b_2(\beta b_2- \delta b_1)}{({\rm det}A)^2} - \frac{\gamma}{{\rm det}A} \leq 0.
 \end{eqnarray*}
\end{proof}

\begin{thm}\label{3x3mpinv_irred}
Let $M\in \mathbb{R}^{3 \times 3}$ be a singular irreducible $F_0$-matrix. Then $M^{\dag} \leq 0.$ 
\end{thm}
\begin{proof}
Let $M=\begin{pmatrix}
 A & b \\
 c^T & 0
\end{pmatrix},$ where $A$ is an $N_0$-matrix, and $b, c$ are nonpositive column vectors. We show that there is a full rank factorization of $M$ given by $M=FG$, where $F^{\dag} \leq 0$ and $G^{\dag} \geq 0$. The proof would then follow, since $M^{\dag}=G^{\dag}F^{\dag}$. 

Again, we consider the three forms that the matrix $A$ can take. Assume that $A$ is given as in \ref{form3}. Then, as argued earlier, there are only two possibilities. In the first case, we have
\begin{center}
$M=\begin{pmatrix}
    0 & \beta & b_1\\
    \gamma & 0 & 0\\
    c_1 & 0 & 0
\end{pmatrix}$
\end{center}
and $M$ has the full rank factorization $M=FG$, where
\begin{center}
$F=\begin{pmatrix}
0 & \beta \\
\gamma & 0 \\
c_1 & 0
\end{pmatrix}$ and $G=\begin{pmatrix}
1 & 0 & 0 \\
0 & 1 & \frac{b_1}{\beta}
\end{pmatrix}.$
\end{center}
It may be verified that 
$$F^{\dag}=\frac{1}{\beta^2(\gamma^2+c_1^2)}
\begin{pmatrix}
0 & \beta^2\gamma & \beta^2c_1 \\
\beta(\gamma^2+c_1^2) & 0 & 0
\end{pmatrix} \leq 0$$ whereas  $$G^{\dag}=\frac{1}{\beta^2+b_1^2}\begin{pmatrix}
\beta^2+b_1^2 & 0 \\
0 & \beta^2 \\
0 & \beta b_1
\end{pmatrix} \geq 0.$$

The other possibility for $M$ is given as
\begin{center}
$M=\begin{pmatrix}
    0 & \beta & 0\\
    \gamma & 0 & b_2\\
    0 & c_2 & 0
\end{pmatrix}.$
\end{center}
Here, $M$ has the full rank factorization $M=FG$, where
\begin{center}
$F=\begin{pmatrix}
0 & \beta \\
\gamma & 0 \\
0 & c_2
\end{pmatrix}$ and $G=\begin{pmatrix}
1 & 0 & \frac{b_2}{\gamma} \\
0 & 1 & 0
\end{pmatrix}.$
\end{center}
Then
$$F^{\dag}=\frac{1}{\gamma(\beta^2+c_2^2)}\begin{pmatrix}
0 & \beta^2+c_2^2 & 0 \\
\beta\gamma & 0 & c_2\gamma
\end{pmatrix} \leq 0$$ and  $$G^{\dag}=\frac{1}{\gamma^2+b_2^2}\begin{pmatrix}
\gamma^2 & 0 \\
0 & \gamma^2+b_2^2 \\
\gamma b_2 & 0
\end{pmatrix} \geq0.$$

Next, assume that $A$ is as in \ref{form1}. Here, $$M=\begin{pmatrix}
    0 & \beta & 0\\
    \gamma & \delta & b_2 \\
    0 & c_2 & 0 
\end{pmatrix}$$
and a full rank factorization $M=FG$ is given by 
\begin{center}
  $F=
  \begin{pmatrix}
      0 & \beta \\
      \gamma & \delta \\
      0 & c_2
  \end{pmatrix},$
  \end{center}
 whereas $G$ is the same matrix as in the previous case. So, $G^{\dag} \geq 0$, while  
  $$F^\dagger= \begin{pmatrix}
    \frac{-\delta \beta}{\gamma(\delta^2+c_2^2)} & \frac{1}{\gamma} & \frac{-\delta c_2}{\gamma(\delta^2+c_2^2)} \\
    \frac{\beta}{\beta^2+c_2^2} & 0  & \frac{c_2}{\beta^2+c_2^2}
  \end{pmatrix}\leq 0.$$ This completes the proof.
  \end{proof}

We illustrate Theorem \ref{3x3mpinv_red} and Theorem \ref{3x3mpinv_irred}, by examples.

\begin{ex}
Let $M=\begin{pmatrix}
    ~~1 &   -2  &  -1\\
    -3 &   ~~1 &   -2\\
    ~~0  &  ~~0  & ~~0
\end{pmatrix}.$ Then $M$ is reducible and is an $F_0$-matrix. We have $$M^\dagger
=\begin{pmatrix}
   ~\frac{1}{15} &  -\frac{1}{5}    &     0 \\
   -\frac{1}{3} &  ~~0   &     0\\
   -\frac{4}{15}  & -\frac{1}{5}   &    0
\end{pmatrix},$$ 
showing that, while the Moore-Penrose inverse is a $Z$-matrix, it is not necessarily nonpositive.
\end{ex}

\begin{ex}
Let $M=\begin{pmatrix}
   ~~0  &  -1  &  -1\\
 -2   & ~~0  &   ~~0\\
    -1   &  ~~0   & ~~0
\end{pmatrix}.$ Then $M$ is an irrreducible $F_0$-matrix and its Moore-Penrose inverse is given by $$M^\dagger=
\begin{pmatrix}
     ~~0  &  -\frac{2}{5} &  -\frac{1}{5}\\
     -\frac{1}{2} &  ~~0  & ~~0\\
     -\frac{1}{2} &  ~~0  & ~~0
\end{pmatrix}\leq 0.$$ 
\end{ex}

\subsection{Complementarity properties}
In this section, we prove some consequences for $F_0$-matrices in the context of the linear complementarity problem. 

Let us recall a result that presents a necessary condition for a matrix to belong to $Q_0 \setminus Q.$ Note that if a matrix has a nonpositive row, then it cannot be a $Q$-matrix. By taking a vector $q$ whose last coordinate is negative, it is immediate to observe that LCP$(A,q)$ does not have a solution. Next, we show that such 

\begin{lem}\label{necessaryQ0}
Let $A$ be a $Q_0$-matrix which is not a $Q$-matrix. Then there exists a vector $0 \neq y \geq 0$ such that $A^Ty\leq 0.$
\end{lem}

\begin{thm}\label{compl}
Let $M$ be a singular $F_0$-matrix, given by $M=\begin{pmatrix}
A & b\\
c^T & 0
\end{pmatrix}.$ Then \\
(1) $M$ is an $R_0$-matrix iff $b \neq 0.$\\
(2) $M$ is not a semimonotone matrix.\\
(3) If $-M$ is a $Q_0$-matrix, then $c=0.$  
\end{thm}
\begin{proof}
(1) Let $b \neq 0.$ Assume that $x:=(y,\lambda)^T$ is chosen such that $x\geq 0$ and $Mx \geq 0$. Then $y \geq 0, \lambda \geq 0$ satisfy the inequalities $$Ay+\lambda b \geq 0~\textit{and}~c^Ty+\lambda d \geq 0.$$
The first inequality, upon premultiplying by $A^{-1} \leq 0$, yields $y+\lambda A^{-1}b\leq 0,$ so that $y \leq -\lambda A^{-1}b \leq 0$, since $A^{-1}b \geq 0.$ Thus $y=0$ and since $0\neq b \leq 0,$ (by the simplified first inequality), we then have $\lambda =0$. Thus $x=0,$ proving that $M$ is an $R_0$-matrix. 

If $b=0,$ then any vector of the form $x:=(0,\lambda)^T$ for $\lambda >0$ is a nonnegative nonzero vector that satisfies $Mx=0$. Hence $M$ is not an $R_0$-matrix.

(2) Let $y \in \mathbb{R}^n$ be chosen such that $y <0$ and $x:=A^{-1}y.$ Then $x >0$ and $(Ax)_k<0$ for every $k$. If we set $z:=(x,0)^T \in \mathbb{R}^{n+1},$ then $z \geq 0$ and $Mz=(Ax,u^Tx)^T.$ If $z_k >0,$ then $1 \leq k \leq n$ and for every such $k$, we have $(Mz)_k=(Ax)_k<0,$ proving that $M$ is not semimonotone.\\
(3) Let $-M$ be a $Q_0$-matrix. Then by Lemma \ref{necessaryQ0}, there exists a vector $y\neq 0,$ $y=(w, \lambda)\geq 0$ such that $-M^Ty \leq 0.$ Expanding this, we obtain $b^Tw\geq 0$ and 
\begin{eqnarray}\label{irrfoequ}
 A^Tw +\lambda c \geq 0.   
\end{eqnarray}
Now, premultiplying \ref{irrfoequ} by $(A^T)^{-1}$ yields $$w +\lambda(A^{-1})^Tc \leq 0,$$
so that $$w\leq -\lambda(A^{-1})^Tc \leq 0.$$ So, $w=0$. Now, \ref{irrfoequ} gives $\lambda c\geq 0$ and so $\lambda c=0$ (since $\lambda$ is nonnegative and $c \leq 0$). If $c$ is a nonzero vector, then we must have $\lambda =0.$ Hence $y=0,$ contradicting $y\neq 0.$ Therefore, $c=0.$
\end{proof}

\begin{rem}
In the proof of  (1), the complementarity condition $x^TMx=0$ has not been used. This means that one has $$x \geq0 ~\textit {and} ~Mx \geq 0 \Longrightarrow x=0.$$ It follows that $-M$ has a {\it left poverse}, that is there exist nonnegative matrices $U, V$ such that $I+V+UM=0.$ We refer the reader to Theorem 8.6, Remark 8.7, Example 8.22 and the relevant citation in \cite{murthyetal} for more details.
\end{rem}

Next, we show that the converse in (3) of the above result is not true.

\begin{ex}
Let $$M=
\begin{pmatrix}
~~0 & -1 & -1\\
-1 & ~~0 & -1\\
~~0 & ~~0 & ~~0
\end{pmatrix}.$$ Then $M$ is a singular $F_0$-matrix. If $q=(-1, 1, 1)^T,$ then $x^0:=(1, 1, 0)^T$ satisfies $-Mx^0 + q \geq 0$ and so $x^0$ is feasible for LCP$(-M,q).$  If possible, let $u=(u_1, u_2, u_3)^T$ be a solution for LCP$(-M,q).$ Then $v:=-Mu+q\geq 0$ and $u^Tv=0.$ Since $q_3$ is positive, the complementarity constraint implies that $u_3=0$. The first two inequalities are $v_1=u_2 -1 \geq 0$ and $v_2=u_1+1 \geq 0$ Now, since $u_2 \geq 1$, from the second complementarity condition, (which is $u_2v_2=0$), it follows that $u_1+1=0$. This contradicts the nonnegativity of $u_1,$ showing that $-M$ is not a $Q_0$-matrix.
\end{ex}

However, we have the following result.

\begin{lem}
Let $M$ be a singular reducible $F_0$-matrix. Let $q=(p,q_n)^T,$ where $p \in \mathbb{R}^n$ such that $p \leq 0$ and $q_n \geq 0$. Then LCP$(-M,q)$ has a solution. 
\end{lem}
\begin{proof}
  We have $M=
  \begin{pmatrix}
    A & b \\
    0 & 0
  \end{pmatrix},$ where $A$ is an $N_0$-matrix and $b\leq 0.$  Define $u:=A^{-1}p \geq 0.$ If $v:=(u,0)^T,$ then $v \geq 0$ and so $$y:=-Mv+q=
\begin{pmatrix}
 -A & -b\\
 ~~0 & ~~0
\end{pmatrix} \begin{pmatrix}
u \\
0
\end{pmatrix} + \begin{pmatrix}
p \\
q_n
\end{pmatrix}=\begin{pmatrix}
-Au+p \\
q_n
\end{pmatrix}=\begin{pmatrix}
0 \\
q_n
\end{pmatrix} \geq 0.$$ Also, $v^Ty=0.$ Hence, $v\in$ SOL$(-M, q).$
\end{proof}

\subsection{Results for circulant matrices}
Here, we present characterizations for $3 \times 3$ circulant matrices to belong to some of the pertinent matrix classes, motivated by the results mentioned earlier. 

\begin{thm}\label{circ}
Let $A$ be a $3 \times 3$ circulant matrix and $z_A=a+ib$, both be as defined earlier. The following hold:\\
(i) Let $A\in {\cal C}^{(-1)}$. Then\\
(1) $A$ is an $F_0$-matrix iff 
\begin{center}
$a>-1 +|b|\sqrt{3},~a\geq \frac{1}{2}$ and $(a-1)^2+b^2<1.$
\end{center}
(2) $A$ is an inverse $N_0$-matrix iff 
\begin{center}
$(a-\frac{1}{2})^2+(b-\frac{\sqrt{3}}{2})^2\geq 1,~(a-\frac{1}{2})^2+(b+\frac{\sqrt{3}}{2})^2\geq 1$ and $a\leq \frac{1}{2}.$    
\end{center}
(3) $A$ is an inverse $F_0$-matrix iff 
\begin{center}
$(a-\frac{1}{2})^2+(b-\frac{\sqrt{3}}{2})^2\geq 1,~(a-\frac{1}{2})^2+(b+\frac{\sqrt{3}}{2})^2\geq 1$ and $a > \frac{1}{2}.$    
\end{center}
(ii) Let $A \in C^{(1)}$. Then $A$ is an inverse $M$-matrix iff  
\begin{center}
$(a-\frac{1}{2})^2+(b-\frac{\sqrt{3}}{2})^2\leq 1$ 
and $(a-\frac{1}{2})^2+(b+\frac{\sqrt{3}}{2})^2\leq 1.$
\end{center}
\end{thm}
\begin{proof}
(i) Recall that a matrix $A$ is of the form \ref{circulantmatrix} iff
$$A=\begin{pmatrix}
\alpha_0
 & \alpha_1 & \alpha_2\\
 \alpha_2  & \alpha_0 & \alpha_1\\
 \alpha_1  & \alpha_2 & \alpha_0 
 \end{pmatrix} \label{fullmatrix}.$$ 
Then, $A \in {\cal C}^{(-1)}$ iff $z_A=\alpha_0+\alpha_1 \omega + \alpha_2 \omega^2,$ where $\alpha_0 + \alpha_1 +\alpha_2=-1.$ We have $a+ib=z_A.$ Thus, $a=\alpha_0-\frac{1}{2}(\alpha_1+\alpha_2)$, which simplifies to $a=-1-\frac{3}{2}(\alpha_1+\alpha_2)$. Also, $b=\frac{\sqrt{3}}{2}(\alpha_1-\alpha_2).$ Hence, 
\begin{eqnarray}\label{alpha1}
\alpha_1=-\frac{a+1}{3}+\frac{b}{\sqrt{3}}
 \end{eqnarray}
 \begin{eqnarray}\label{alpha2}
     \alpha_2=-\frac{a+1}{3}-\frac{b}{\sqrt{3}}.
 \end{eqnarray}
(1) For $A$ to be an $F_0$-matrix, the following inequalities are necessary and sufficient: 
\begin{center}
$\alpha_0 \geq 0,~\alpha_1<0, ~\alpha_2<0$ and $\alpha_0^2<\alpha_1 \alpha_2.$
\end{center}
The nonpositivity of $\alpha_1$ and $\alpha_2$ yield the inequality $a>-1 +|b|\sqrt{3}.$ The nonnegativity of $\alpha_0$ is equivalent to $a\geq \frac{1}{2}.$ The inequality $\alpha_0^2 <\alpha_1 \alpha_2$ reduces to 
 \begin{eqnarray}\label{alpha12}
(\alpha_1+\alpha_2)^2+2(\alpha_1+\alpha_2)-\alpha_1 \alpha_2+1<0.
 \end{eqnarray}
This is the requirement $(a-1)^2+b^2<1.$

(2) The matrix $A$ is an inverse $N_0$-matrix iff 
\begin{center}
$\alpha_2^2 - \alpha_0 \alpha_1 \geq 0$ $\alpha_1^2-\alpha_0 \alpha_2 \geq 0$ and $\alpha_0 \leq 0.$
\end{center}
The nonnegativity requirement on the numbers $\alpha_2^2-\alpha_0 \alpha_1$ and  $\alpha_1^2-\alpha_0 \alpha_2 $ yield the inequalities 
\begin{center}
$(a-\frac{1}{2})^2+(b-\frac{\sqrt{3}}{2})^2\geq 1$ and $(a-\frac{1}{2})^2+(b+\frac{\sqrt{3}}{2})^2\geq 1,$    
\end{center}
respectively. The nonpositivity of $\alpha_0$ is equivalent to $a\leq \frac{1}{2}.$

(3) The matrix $A$ is an inverse $F_0$-matrix iff 
\begin{center}
$\alpha_2^2 - \alpha_0 \alpha_1 \geq 0,$ $\alpha_1^2-\alpha_0 \alpha_2 \geq 0$ and $\alpha_0 > 0.$    
\end{center}
The nonnegativity requirements $\alpha_2^2-\alpha_0 \alpha_1 \geq 0$ and $\alpha_1^2-\alpha_0 \alpha_2 \geq 0$ reduce to 
\begin{center}
$(a-\frac{1}{2})^2+(b-\frac{\sqrt{3}}{2})^2\geq 1$ and $(a-\frac{1}{2})^2+(b+\frac{\sqrt{3}}{2})^2\geq 1,$ respectively. 
\end{center}
The condition $\alpha_0 >0$ is equivalent to $a > \frac{1}{2}.$

(ii) For a matrix $A$ of the form \ref{fullmatrix} to be in ${\cal C}^{(1)}$ it is necessary and sufficient that $z_A=\alpha_0+\alpha_1 \omega + \alpha_2 \omega^2,$ where $\alpha_0 + \alpha_1 +\alpha_2=1.$ For this matrix class, we have $a=\alpha_0-\frac{1}{2}(\alpha_1+\alpha_2)$, which simplifies to $a=1-\frac{3}{2}(\alpha_1+\alpha_2)$. Also, $b=\frac{\sqrt{3}}{2}(\alpha_1-\alpha_2).$ Thus, 
\begin{eqnarray}\label{alpha_1}
\alpha_1=\frac{1}{3}(1-a)+\frac{b}{\sqrt{3}}
 \end{eqnarray}
 \begin{eqnarray}\label{alpha_2}
     \alpha_2=\frac{1}{3}(1-a)-\frac{b}{\sqrt{3}}.
 \end{eqnarray} 
 Now, $A$ is an inverse $M$-matrix iff 
 \begin{center}
 $\alpha_2^2-\alpha_0 \alpha_1 \leq 0$ and $\alpha_1^2-\alpha_0 \alpha_2 \leq 0.$    
 \end{center}
The nonpositivity of $\alpha_2^2-\alpha_0 \alpha_1$ is the same as the inequality $(a-\frac{1}{2})^2+(b-\frac{\sqrt{3}}{2})^2\leq 1,$ while the nonpositivity of $\alpha_1^2-\alpha_0 \alpha_2$ is equivalent to $(a-\frac{1}{2})^2+(b+\frac{\sqrt{3}}{2})^2\leq 1.$
\end{proof}

\subsection{Inverse $F_0$-matrices}
In the last part of this work, we present a condition for an invertible matrix to be an inverse $F_0$-matrix, in Theorem \ref{d_n+1f0}. Even though this result follows as a particular case of \cite[Theorem 3.12]{nabben1}, our proof is more elementary.

Let us recall that a matrix $M$ is said to be {\it totally nonpositive} if all its minors, of any order, are nonpositive. For a matrix $A$, we denote by $A(i:j)$, the submatrix of $A$ obtained by deleting the row indexed by $i$ and the column indexed by $j.$ 

\begin{thm}\label{charinvf0}
{\rm
 Let $M\in \mathbb{R}^{(n+1) \times (n+1)}$ satisfy the conditions that it has at least one positive diagonal entry and all of whose minors of order at least two, are nonpositive. Then $M$ is an inverse $F_0$-matrix iff $\rm{det}(M)<0$ and $\rm{det} M(i:j)=0$ whenever the indices $i,j$ satisfy $i+j=2k,$ $i\neq j$, for any positive integer $k$.}
\end{thm}
\begin{proof}
{\rm
 Suppose that $M$ is an inverse $F_0$-matrix, so that ${\rm det}M<0$. Set $M^{-1}=(m_{ij})$. Then $$m_{ij}=(-1)^{i+j}{\rm det} M(i:j)/{\rm det}M.$$

Since $m_{ij}\leq 0$ whenever $i\neq j$, and since the determinants of the submatrices $M(i:j)$ are nonpositive for all $i, j$, we conclude that ${\rm det} M(i:j) = 0$ for any $i\neq j$ such that $i+j$ is even.

 For the converse, assume that $\rm{det}(M)<0$ and $\rm{det} M(i:j)=0$ whenever $i,j$ are such that $i\neq j$ and $i+j=2k$, for some positive integer $k$. This means that, for $i\neq j$, if $i+j$ is even, then $m_{ij}=0$. Also, when $i+j$ is odd, then $m_{ij}\leq 0$. Therefore, $M^{-1}$ is a $Z$-matrix. The rest of the proof follows, by appealing to Theorem \ref{charf0}.
 }
\end{proof}
We need an intermediate result, which in turn, uses another proposition. This is stated next.

\begin{thm}{\rm \cite[Theorem 2.3]{gjohn2}\label{totallynonpositive}}
  If $M\in \mathbb{R}^{n \times n}$ is a matrix of type $D,$ and $a_{n}<0,$ then all minors of $M$ are nonpositive.
\end{thm}

\begin{thm}\label{dn+1}
{\rm 
  Let $M\in \mathbb{R}^{(n+1) \times (n+1)}$ be a matrix of type $D,$ where $a_{n+1}>0$ and $a_{n}<0.$ Then all minors of $M$ of order at least two, are nonpositive.
  }
\end{thm}
\begin{proof}
{\rm
  The matrix $M$ is given by  $$D_{n+1}=
  \begin{pmatrix}
   a_1 & a_1 & a_1 & .... & a_1 & a_1\\
   a_1 & a_2 & a_2 & .... & a_2 & a_2\\
   \vdots & & \vdots &... & \vdots & \vdots\\
   a_1 & a_2 & a_3 & .... & a_{n} & a_n\\
   a_1 & a_2 & a_3 & .... & a_n & a_{n+1}
  \end{pmatrix},$$ 
which we rewrite using block matrices as $$M=
  \begin{pmatrix}
      D_n & b \\
      b^T & a_{n+1}
  \end{pmatrix},$$ where $b=(a_1, a_2,...,a_n)^T$ has all its entries negative. Since $M$ is invertible and $M/D_{n}\geq 0$, we have ${\rm det}(M)={\rm det}(A)(a_{n+1}-b^TA^{-1}b)<0$. Consider $M(i:j)$ (which is a submatrix of $D_{n+1}$). To compute the determinant of $M(i:j,)$ we consider the following cases:

$\underline{Case ~1}:$ $i=j=n+1.$ Then $M(i:j)=D_n,$ and so by Theorem \ref{totallynonpositive}, ${\rm det}(M(i:j))<0.$

$\underline{Case~ 2}:$ $M(i:j)$ contains both the last row and the last column of $M$. The same argument used earlier to show that ${\rm det}(M)<0$, applies. We get ${\rm det} M(i:j) \leq 0$.

$\underline{Case~3}:$  $1\leq i \leq n, j=n+1.$ Here, either $M(i:j)=D_n$ or $M(i:j)$ is a submatrix of $D_{n+1}$ with two identical rows. So, either ${\rm det} M(i:j)={\rm det} (D_n)<0$ or ${\rm det} M(i:j)=0.$

$\underline{Case~4}:$ This is the same as the previous case, with the roles of $i$ and $j$ interchanged. Hence, once again, either ${\rm det}M(i:j)<0$ or ${\rm det}M(i:j)=0.$

An entirely similar argument as above may be used to show that every minor of order at least two, is nonpositive. This completes the proof.
}
\end{proof}

As a consequence of this and Theorem \ref{charinvf0} and \ref{dn+1}, we have the result that was stated earlier.

\begin{thm}\label{d_n+1f0}
 Let $M\in \mathbb{R}^{(n+1) \times (n+1)}$ be a type $D$ matrix, with $a_{n+1}>0$ and $a_{n}<0.$ Then $M^{-1}$ is a tridiagonal $F_0$-matrix.   
\end{thm}
\begin{proof}
First, note that the diagonal entry $a_{n+1} >0$. Also, by Theorem \ref{dn+1}, we have $\rm det(M) <0.$ Thus, the hypotheses of Theorem \ref{charinvf0} are satisfied. To conclude that $M$ is an inverse $F_0$-matrix, we only need to prove that ${\rm det} M(i:j)=0$ for all $i,j$ with $i \neq j$ such that $i+j=2k$, where $k$ is a positive integer. This condition, however follows from the fact that $M^{-1}$ is a tridiagonal matrix \cite[Theorem 8]{stuartczech}.
\end{proof}

We illustrate Theorem \ref{d_n+1f0}, by an example.

\begin{ex}
{\rm
 Consider the matrix $$M=
 \begin{pmatrix}
   -3 &   -3   & -3  &  -3\\  
   -3  &  -2  &  -2  &  -2 \\
   -3 &   -2  &  -1  &  -1 \\
   -3  &  -2   & -1   &  ~1
 \end{pmatrix}.$$ Then $M$ is a type $D$ matrix, satisfying the conditions of Theorem \ref{d_n+1f0}. It may verified that $M^{-1}$ is the tridiagonal matrix given by $$M^{-1}=
 \begin{pmatrix}
     ~~\frac{2}{3}  & -1      &   ~~0     &   ~ ~0\\
     -1 &   ~~2 &  -1    &     ~~0\\
     ~~0 &  -1 &   ~~\frac{3}{2} &  -\frac{1}{2}\\
     ~~0    &   ~~0  & -\frac{1}{2}  &  ~~\frac{1}{2}
 \end{pmatrix}=\begin{pmatrix}
    A & b\\
    c^T & d
\end{pmatrix},$$ 
where $A=\begin{pmatrix}
  ~~\frac{2}{3}  & -1      &   ~~0 \\  
   -1 &   ~~2 &  -1        \\
    ~~0 &  -1 &   ~~\frac{3}{2}
\end{pmatrix}$ is a $N_0$-matrix, $b=c=(0, 0, -\frac{1}{2})^T \leq 0,$ and $d=\frac{1}{2} \neq c^TA^{-1}b.$ Hence, $M^{-1}$ is an $F_0$-matrix.
}
\end{ex}

\section{Concluding remarks}
We have looked at the problem of extending results known for a class of invertible $Z$-matrices to the case, when the matrices are singular. This study has opened up the problem of seeking  generalizations involving the bigger class of all $Z$-matrices. Let us just point to a couple of interesting questions that were not investigated here. 
If $A=
 \begin{pmatrix}
   ~~3 &   -2   & -2  \\  
   -2  &  -1  &  -1  \\
   -2  &  -1   & ~~0  
 \end{pmatrix},$ then $A$ is not an $F_0$-matrix (since it has a negative diagonal) and $A^{-1}=\frac{1}{7}
 \begin{pmatrix}
   ~~1 &   -2   & ~~0  \\  
   -2  &  ~~4  &  -1  \\
   ~~0  &  -1   & ~~1  
 \end{pmatrix}.$ The comparison matrix of $A$ is given by 
 ${\cal M}_A=
 \begin{pmatrix}
   ~~3 &   -2   & -2  \\  
   -2  &  ~~1  &  -1  \\
   -2  &  -1   & ~~0  
 \end{pmatrix}.$ ${\cal M}_A$ is nonsingular and 
${{\cal M}_A}^{-1}=\frac{1}{15}
 \begin{pmatrix}
   ~~1 &   -2   & -4  \\  
   -2  &  ~~4  &  -7  \\
   -4  &  -7   & ~~1  
 \end{pmatrix}.$ We then have $${{\cal M}_A}^{-1} \leq |A^{-1}|,$$ an inequality that goes in the reverse direction to that of Ostrowski (Theorem \ref{ostr})! This motivates the question: Is this true, in general for the class of $Z$-matrices, which have at least one negative diagonal entry? Since the comparison matrix of an $F_0$-matrix is itself, the above inequality holds trivially for $F_0$-matrices, since for such matrices, we have 
 $${{\cal M}_A}^{-1} ={A}^{-1} \leq |A^{-1}|.$$

 The second question is whether the result of Fan (and its converse, namely that $A-I$ is an invertible $M$-matrix iff both $A$ and $I-A^{-1}$ are invertible $M$-matrices), holds good for $F_0$-matrices? More generally, is it true for the class of $Z$-matrices?\\

{\bf Acknowledgements:}

This work was done when the second author was visiting the Department of Mathematics, University of California, Santa Barbara during the Winter and Spring of 2023. The author records his thanks to Dr. Maria Isabel for the excellent hospitality that was extended to him. The first author acknowledges  funding  received  from  the Prime  Minister’s  Research  Fellowship  (PMRF),  Ministry  of  Education, Government of India, for carrying out this work.


\newpage
\begin{thebibliography}{10} 

\bibitem{bg}
A. Ben-Israel and T.N.E. Greville, {\it Generalized Inverses - Theory and Applications}, CMS Books in Mathematics vol. 15. Springer-Verlag, New York, 2003.

\bibitem{berpl}
A. Berman and R.J. Plemmons, {\it Nonnegative Matrices in the Mathematical Sciences},  Classics in Applied Mathematics, SIAM, Philadelphia, PA, 1994.

\bibitem{berpl1}
A. Berman and R.J. Plemmons, {\it Monotonicity and the generalized inverse}, SIAM J. Appl. Math., {\bf 22} (1972) 155-161.

\bibitem{manamiinequalities}
M. Chatterjee and K.C. Sivakumar, {\it Inequalities for group invertible H-matrices}, Linear Algebra Appl., {\bf 576} (2019) 158-180.

\bibitem{manamiinterval}
M. Chatterjee and K.C. Sivakumar, {\it Intervals of $H$-matrices and inverse $M$-matrices}, Linear Algebra Appl., {\bf 614} (2021) 24-43.

\bibitem{chen}
Y. Chen, {\it Notes on $F_0$-matrices}, Linear Algebra Appl., {\bf 142}, (1990) 167-172.

\bibitem{fan}
K. Fan, {\it Some matrix inequalities}, Abh. Math. Se. Univ. Hamburg {\bf 29}, (1966) 185-196.

\bibitem{fiedmarneu}
M. Fiedler, T.L. Markham, and M. Neumann, {\it Classes of products of $M$-matrices and inverse $M$-matrices,} Linear Algebra Appl., {\bf 52/53} (1983) 265-287.

\bibitem{fiedlermark}
M. Fiedler, T.L. Markham, {\it A classification of matrices of class $Z$}, Linear Algebra Appl., {\bf 173} (1992), 115-124.


\bibitem{john}
C.R. Johnson, {\it Inverse $M$-matrices}, Linear Algebra Appl., {\bf 47} (1982) 195-216.

\bibitem{gjohn1}
G.A. Johnson, {\it A generalization of $N$-matrices}, Linear Algebra Appl., {\bf 48}, (1982) 201-217.

\bibitem{gjohn2}
G.A. Johnson, {\it Inverse $N_0$-matrices}, Linear Algebra Appl., {\bf 64}, (1985) 215-222.


\bibitem{Johnsmith}
C.R. Johnson, R.L. Smith, {\it Inverse $M$-matrices II}, Linear Algebra Appl., {\bf 435} (2011) 195-216.
		
\bibitem{jst}
C.R. Johnson, R.L. Smith and M.J. Tsatsomeros, Matrix Positivity, Cambridge University Press, Cambridge, 2020.

\bibitem{kalkcs1}
A. Kalauch, S. Lavanya and K.C. Sivakumar, {\it Singular irreducible $M$-matrices revisited},  Linear Algebra Appl., {\bf 565} (2019) 47-64.

\bibitem{kalkcs2}
A. Kalauch, S. Lavanya and K.C. Sivakumar, {\it Matrices whose group inverses are $M$-matrices}, Linear Algebra Appl., {\bf  614} (2021) 44-67.


\bibitem{markham}
T.L. Markham, {\it Nonnegative matrices whose inverses are $M$-matrices,} Proc. Amer. Math. Soc. {\bf 36(2)} (1972) 326-330.

\bibitem{debkcs}
D. Mishra and K.C. Sivakumar, {\it
Generalizations of matrix monotonicity and their
relationships with certain subclasses of proper splittings}, Linear Algebra Appl., {\bf 436} (2012) 2604-2614.

\bibitem{murthyetal}
G.S.R. Murthy, K.C. Sivakumar and P. Sushmitha, {\it T. Parthasarathy's contributions to complementarity problems: a survey}, Ann. Oper. Res., {\bf 287} (2020) 867-894.


\bibitem{nabben1}
R. Nabben, {\it $Z$-matrices and inverse $Z$-matrices}, Linear Algebra Appl., {\bf 256} (1997), 31-48.

\bibitem{nabben3}
R.Nabben, {\it On relationships between several classes of Z-matrices, M-matrices and nonnegative matrices}, Linear Algebra Appl., {\bf 421} (2007), 417-439.


\bibitem{nabben2}
R. Nabben ahd R.S. Varga, {\it On classes of inverse $Z$-matrices}, Linear Algebra Appl., {\bf 223/224} (1995), 521-552.


\bibitem{neu}
A. Neumaier, {\it Interval methods for systems of equations}, Encyclopedia of Mathematics and its Applications, {\bf 37}  Cambridge University Press, Cambridge, 1990.

\bibitem{ost}
A. Ostrowski, {\it Über die determinanten mit überwiegender Hauptdiagonale} (German), Comment. Math. Helv., {\bf 10} (1937) 69-96. 

\bibitem{samirtsatkcs}
S. Mondal, K.C. Sivakumar and M.J. Tsatsomeros, {\it New results on $M$-matrices, $H$-matrices and their inverse classes}, Electron. J. Linear Algebra, {\bf 38} (2022) 729-744.


\bibitem{kcsjota}
K.C. Sivakumar, {\it Proof by verification of the Greville/Udwadia/Kalaba formula for the Moore-Penrose inverse of a matrix}, J. Optim. Theory Appl., {\bf 131} (2006) 307-311.

\bibitem{smith1}
R.L. Smith, {\it On the spectrum of $N_0$-matrices}, Linear Algebra Appl., {\bf 83}, (1986) 129-134.

\bibitem{smith2}
R.L. Smith, {\it Some notes on $Z$-matrices}, Linear Algebra Appl., {\bf 106}, (1988) 219-231.

\bibitem{stuart}
J.L. Stuart, {\it Special families of matrices - a talk in honor of Miroslav Fiedler}, Linear Algebra Appl., {\bf 439}, (2013) 830-835.

\bibitem{stuartczech}
J.L. Stuart, {\it Nested matrices and inverse M-matrices}, Czechoslovak Math. J., {\bf 65} (2015) 537-544.

\end{thebibliography}
\end{document}